\documentclass[12pt,a4paper]{article}
\usepackage[centertags]{amsmath}
\usepackage{amsfonts}
\usepackage{calc}
\usepackage{tikz}
\usetikzlibrary{decorations.markings}
\tikzstyle{vertex}=[circle, draw, inner sep=0pt, minimum size=6pt]

\usepackage{graphics}
\usepackage{amssymb}
\usepackage{amsthm}
\usepackage{newlfont}
\usepackage{epsfig}
\usepackage{amssymb}
\usepackage{amsmath}
\usepackage{amsthm}
\usepackage{graphicx}
\usepackage{makeidx}
\theoremstyle{plain}
\newtheorem{thm}{Theorem}[section]
\newtheorem{cor}[thm]{Corollary}

\theoremstyle{definition}

\theoremstyle{remark} \tolerance=10000 \hbadness=10000
\def \ni{\noindent}

\author{
	Jismy Varghese \footnote{E-mail : kvjismy@gmail.com}\\ School of Computer Science\\
	DePaul Institute of Science and Technology\\ Angamaly - 683 573\\ \vspace{0.2cm} Kerala, India. \and
	Anu V.\footnote{E-mail : anusaji1980@gmail.com}\\ Department of Mathematics\\
	St. Peter's College\\ Kolenchery -
	682 311\\  Kerala, India.\\ \and
Aparna Lakshmanan S.	\footnote{E-mail : aparnaren@gmail.com}\\
	Department of Mathematics\\
	St. Xavier's College for Women\\Aluva -
	683 101\\\vspace{0.2cm} Kerala, India.}

\title{\textbf{Italian Domination and Perfect Italian Domination on Sierpi\'{n}ski Graphs}}

\date{}
\begin{document}
	
	\maketitle
	
	\begin{abstract}
			 \ni An Italian dominating function (IDF) of a graph G is a function $ f: V(G) \rightarrow \{0,1,2\} $ satisfying the condition that for every $ v\in V $ with $ f(v) = 0$, $\sum_{ u\in N(v)} f(u) \geq 2. $ The weight of an IDF on $G$ is the sum $ f(V)= \sum_{v\in V}f(v) $ and the Italian domination number, $ \gamma_I(G) $, is the minimum weight of an IDF. An IDF is a perfect Italian dominating function (PID) on $G$, if for every vertex $ v \in V(G) $ with $ f(v) = 0 $ the total weight assigned by $f$ to the neighbours of $ v $ is exactly 2, i.e., all the neighbours of $u$ are assigned the weight 0 by $f$ except for exactly one vertex $v$ for which $ f(v) = 2 $ or  for exactly two vertices $v$ and $w$ for which $ f(v) = f(w) = 1 $. The weight of a PID- function is $f(V)=\sum_{u \in V(G)}f(u)$. The perfect Italian domination number of $G$, denoted by $ \gamma^{p}_{I}(G), $ is the minimum weight of a PID-function of $G$. In this paper we obtain the  Italian domination number and perfect Italian domination number of Sierpi\'{n}ski graphs.
		   \\
		
		\ni {\bf Keywords:} Italian Domination, Perfect Italian Domination, Sierpi\'{n}ski graph.  \\
		
		\ni {\bf AMS Subject Classification:}   05C69, 05C76.\\
	\end{abstract}

\section{Introduction}

Let $G$ be a simple graph with vertex set $V(G)$ and edge set $ E(G) $. If there is no ambiguity in the choice of $G$, then we
write $V(G)$ and $E(G)$ as $V$ and $E$ respectively. The number of vertices and edges of the graph $G$ is denoted by $n(G)$ and $m(G)$ respectively. The open
neighbourhood of  a vertex $v\in V$ is the set $N(v)=\{u: uv \in
E\}$ and the vertices in $N(v)$ are called the neighbours of $v$. $|N(v)|$ is called the degree of the vertex $v$ in $G$ and is denoted by $d_{G}(v)$, or simply $d(v)$. A subset $ S \subseteq V$ of vertices in a graph is called a dominating set if every $ v\in V$ is either an element of $S$ or is adjacent to an element of $S$ \cite{Tha}. The domination number $\gamma(G) $ is the minimum cardinality of a dominating set of $G$. \\

  An Italian dominating function (IDF) of a graph $G$ is a function $ f: V(G) \rightarrow \{0,1,2\} $ satisfying the condition that for every $ v\in V $ with $ f(v) = 0$ $ \sum_{ u\in N(v)} f(u) \geq 2 $. i.e., either $v$ is adjacent to a vertex $u$ with $ f(u) = 2 $ or to at least two vertices $x$ and $y$ with $ f(x) = f(y) = 1. $ The weight of an Italian dominating function is $f(V)=\sum_{u \in V}f(u)$. The Italian domination number $ \gamma_I(G) $ is the minimum weight of an Italian dominating function. The Italian dominating function with weight $ \gamma_I(G) $ is called a $ \gamma_I $-function \cite{Mustha}. Also the sum of the weights of the vertices of $ H $ is denoted by $ f(H) $, where $ H $ is any subgraph of $ G $. i.e., $f(H)= \sum_{u \in V(H)}f(u)$. \\

  An Italian dominating function is a perfect Italian dominating function, abbreviated PID-function, on $G$ if for every vertex $ v \in V(G) $ with $ f(v) = 0 $ the total weight assigned by $f$ to the neighbours of $ v $ is exactly 2, i.e., all the neighbours of $u$ are assigned the weight 0 by $f$ except for exactly one vertex $v$ for which $ f(v) = 2 $ or  for exactly two vertices $v$ and $w$ for which $ f(v) = f(w) = 1 $. The weight of a PID- function is $f(V)=\sum_{u \in V(G)}f(u)$. The perfect Italian domination number of $G$, denoted by $ \gamma^{p}_{I}(G), $ is the minimum weight of a PID-function of $G$ \cite{Twh}.\\

 The study of Italian domination was introduced by M. Chellai, T. W. Haynes, S. T. Hedetniemi and A. A. Mcrae in \cite{Mustha}. It is proved that if $ G $ is a connected graph of order $ n \geq 3 $, then $ \gamma_I(G) \leq \frac{3n}{4} $. If $ G $ has minimum degree at least 2, then $ \gamma_I(G) \leq \frac{2n}{3}$. The connected graphs achieving equality in these bounds were studied in \cite{Hay}. Nordhaus-Gaddum inequalities for Italian domination number is also proved in \cite{Hay}. Italian domination in trees was discussed in \cite{Hen}. In \cite{Haji} the authors studied the graphs with equal domination number and Italian domination number. A. Poureidi and N. J. Rad showed that the associated decision problem for Italian domination is NP-complete even when restricted to planar graphs. They gave a linear algorithm that computes the Italian domination number of a given unicyclic graph \cite{Pour}. Italian domination number of generalized Petersen graphs $ P(n,3) $ was studied in \cite{Gao}. A bagging approach and a partitioning approach  to investigate the Italian domination number of cartesian product of cycles and paths $ C_n \Box P_n $ was done in \cite{Gao2}. They also determined the exact value of the Italian domination number of $ C_n \Box P_3 $ and $ C_3 \Box P_n $ and bounds for $\gamma_I( C_n \Box P_m) $ for $ m,n \geq 4 $.   Corona operator on Italian domination was studied in \cite{Jis}.

\section{Sierpi\'{n}ski Graphs}
 Let $ G=(V,E) $ be a non-empty graph of order $ n\geq 2, $ and $t$ a positive integer. Let $ V^t $ be the set of words of length $t$ on alphabet $V$. The letters of a word $u$ of length $t$ are denoted by $ u_1u_2...u_t$.
 The graph $ S(K_n,t), t\geq 1, $ ($ S(t,n) $ in their notation) was introduced by Klav\v{z}ar and Milutinovi\'{c} in \cite{SKalv}.  $ S(K_n,t)$ has vertex set $ V^t $ and $ \{u,v\} $ is an edge if and only if there exists $ i\in \{1,2,...,t\} $ such that:\\
 (i) $ u_j = v_j,$ if $j<i; $ (ii) $ u_i \neq v_i; $ (iii) $ u_j = v_i$ and $ v_j = u_i,$ if $ j>i. $ \\

 Later, those graphs have been called Sierpi\'{n}ski graphs in \cite{Klav}. Figure 1 and Figure 2  illustrate  Sierpi\'{n}ski graphs $ S(K_5,1),\ S(K_5,2) $ and $ S(K_5,3) $. The vertices of the form $ uuu...u $ are called extreme vertices of $ S(K_n,t) $. Note that for any $ t \geq 2 $, $ S(K_n,t) $ has $ n $ extreme vertices and the extreme vertex $ uuu...u $ has degree $n-1$.\\

 \begin{figure}[h] \label{Fig1}
 	\centering
 	\includegraphics[width=8.5cm]{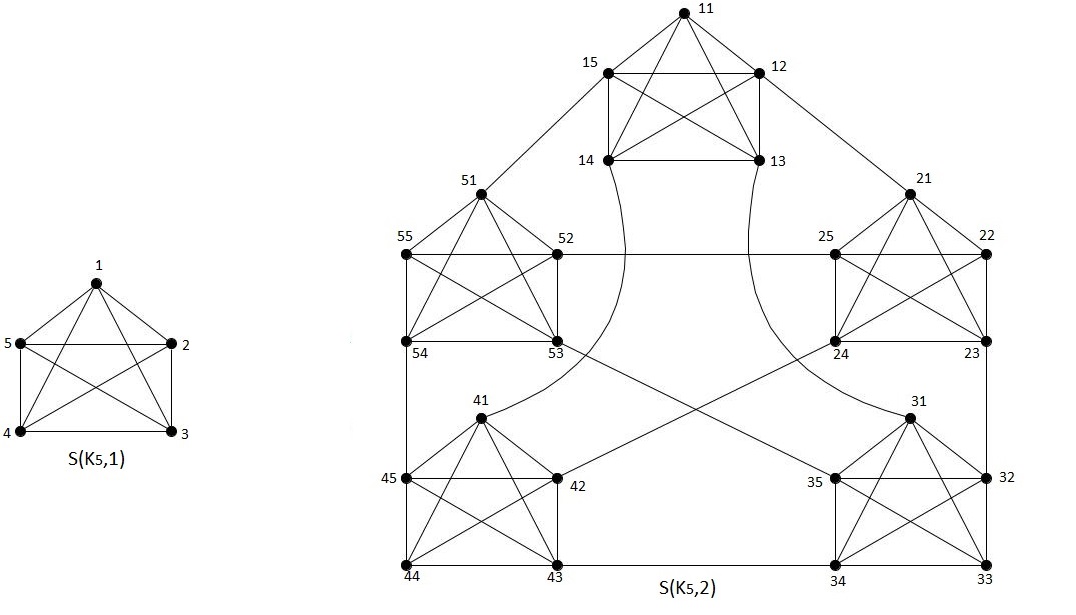}\\
 	\caption{Sierpi\'{n}ski graphs $S(K_5,t), t=1,2$.\label{overflow} }	
\end{figure}

\begin{figure}[h] \label{Fig2}
	\centering
 	\includegraphics[width=10.5cm]{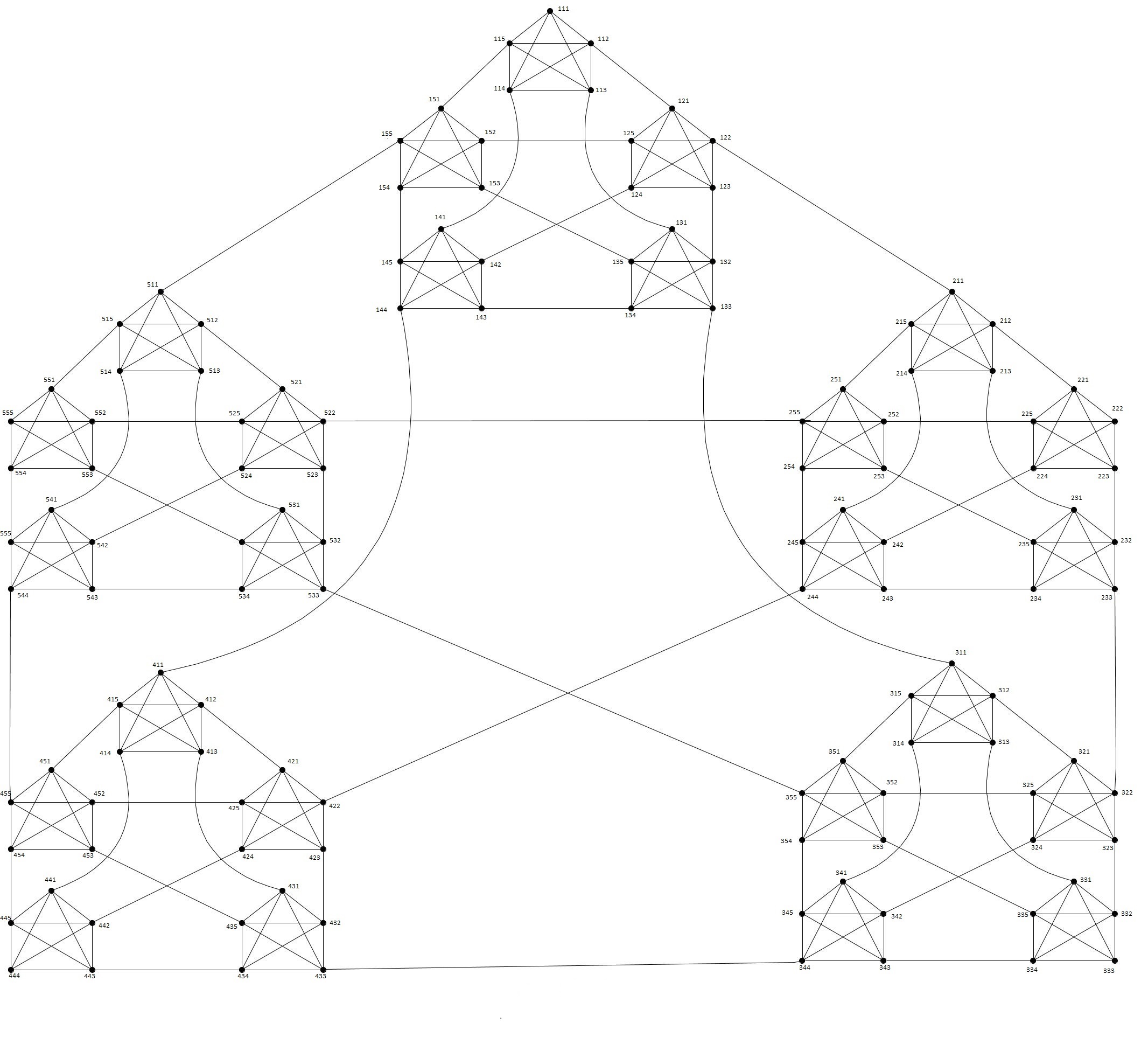}\\
 	\caption{Sierpi\'{n}ski graph $S(K_5,3)$.  }
 \end{figure}
  $S(K_n, t)$ can be constructed
recursively from $K_n$ as follows: Take  $K_n$ as $S(K_n, 1)$. To construct $S(K_n, t)$ for $t\geq 2$, take $n$
copies of  $S(K_{n}, t - 1)$ and add the letter $x$ at the beginning of each label of the vertices belonging to the copy of $S(K_n, t-1)$ corresponding to $x$. Then  add an edge between vertex $xyy . . . y$ and vertex $yxx . . . x$. Domination number, Roman domination number and double Roman domination number of Sierpi\'{n}ski graphs were studied in \cite{Ram}, \cite{Anu} and  \cite{Liu}. To understand more about Sierpi\'{n}ski graphs readers are requested to go through the survey paper \cite{Hin} appeared in $2017$.\\

  For any graph theoretic terminology and notations not mentioned here, the readers may refer to \cite{Bal}.
 \section{Main Results}
 In this section, the exact value of the Italian domination number  of the Sierpi\'{n}ski graph, $ \gamma_I(S(K_n,t))$ is given. For $ n=2 $, $ S(K_n,t) = P_{2^t} $ and in \cite{Ali} it has been proved that $ \gamma_I(P_n)= \lceil \frac{n+1}{2}\rceil $. For $ t=2$ and $ n\geq 2 $ we already have the following theorem.\\
 \begin{thm} \label{thm3} \cite{Jism}
 	The Italian domination number of the Sierpinski graph $S(K_n, 2)$ is $ 2n-1.$
 \end{thm}
For $ t=3 $, we prove the following theorem, which can then be generalized to $ t > 3. $
 \begin{thm} \label{thm1}
The Italian domination number of the Sierpi\'{n}ski graph $S(K_n, 3)$ is $ 2n(n-1) $, for $ n\geq 3 $.
\end{thm}
\begin{proof}
	Let $ V(K_n) = \{ v_1,v_2,...,v_n\} $. Then $ S(K_n,3)  $ has the  vertex set $\{ v_iv_jv_k: i,j,k \in \{1,2,...,n\} \}$. Note that by the definition, there are three types of adjacencies in $S(K_n,3)$.
	 \begin{enumerate}
		\item [$\bullet$] $ \{ \{v_iv_jv_j, v_jv_iv_i\}: i,j \in \{1,2,...,n\}, i\neq j \} $.
		\item [$\bullet$ ]$ \{ \{v_iv_jv_k, v_iv_kv_j \}: i,j,k \in \{1,2,...,n\}, j\neq k \} $.
		\item [$\bullet$] $\{ \{v_iv_jv_k, v_iv_jv_l\}: i,j,k,l \in \{1,2,...,n\}, k\neq l \} $.
	\end{enumerate}
  Let $ S_i(K_n,2) $ denote the $ i^{th} $ copy of $ S(K_n,2) $ in $ S(K_n,3) $ and $ S_{ij}(K_n) $ denote the $ j^{th} $ copy of $ K_n $ in $ S_i(K_n,2) $ for $ i,j=1,2,...,n $.  Define an Italian dominating function on $ S(K_n,3) $ as follows.\\
	\[
	f(v)=
	\begin{cases}
	1;\ v=v_iv_jv_{i-1},\  i,j \in \{1,2,...,n\}, and \\ \hspace{6mm} v=v_iv_jv_{i+1}, \ i,j \in \{1,2,...,n\},\ j \neq i-1,i+1,\\
	0;\ otherwise.
	\end{cases}
	\]
	
	For any $ f(v_iv_jv_k)=0 $, $ v_iv_jv_k $ is adjacent to $ v_iv_jv_{i-1} $ and $ v_iv_jv_{i+1}$, for $k \neq i - 1$ and $i+1$.  By the definition of $f$, $ f(v_iv_jv_{i-1}) =1$ and $ f(v_iv_jv_{i+1})=1 $, so that $v_iv_jv_k$ is Italian dominated. Note that, $f(v_iv_jv_{i-1}) = 1$, so that $f(v_iv_jv_k) = 0$ implies $k \neq i - 1$. If $ k=i+1 $, then $f(v_iv_jv_k)=0$ implies $j= i-1 $ or $ i+1 $. If $ j=i-1 $ the vertex $ v_iv_{i-1}v_{i+1} $ is adjacent to the vertex $ v_iv_{i+1}v_{i-1} $ and $ f(v_iv_{i+1}v_{i-1})=1 $. If $ j=i+1 $ the vertex $ v_iv_{i+1}v_{i+1} $ is adjacent to $ v_{i+1}v_iv_i $ and $ f(v_{i+1}v_iv_i)=1 $. Therefore, $ f $ is an IDF of $ S(K_n,3) $ and $ f(V)=2n(n-1). $ So $ \gamma_I(S(K_n,3)) \leq 2n(n-1). $\\
	
	To prove the reverse inequality, we claim that  for any $\gamma_I$-function $f$ of $S(K_n,3)$,  $f( S_i(K_n,2))\geq 2n-2 $, for $i=1,2,\ldots,n$. If possible assume that there exists an $\gamma_I$-function $ f $ such that weight of a copy, say $ p^{th} $ copy, of $ S(K_n,2) $ is at most $ 2n-3 $. Then two cases arise.
	\begin{enumerate}
		\item [1.] There exists a copy, say $ S_{pa}(K_n) $, of $ K_n $ in $ S_p(K_n,2) $ with weight 0,  another copy, say $ S_{pb}(K_n) $, of $ K_n $ with weight 1   and all other copies of $ K_n $ with weight 2.
		\item [2.] There exist three copies, say $ S_{pa}(K_n) $, $ S_{pb}(K_n) $, $ S_{pc}(K_n) $, of $ K_n $ with weight 1  and all other copies of $ K_n $ with weight 2.
	\end{enumerate}
		\ni {\bf \emph{Case 1:}}\\
		
		\ni Let $ f(S_{pa}(K_n))=0 $, $ f(S_{pb}(K_n))=1 $ and $ f(S_{pj}(K_n))=2, $ for all $ j=1,2,...,n $ and $ j \neq a,b $. Since $ S_{pa}(K_n) $ has weight 0, all the vertices $ v_pv_av_j $ has weight 0. To Italian dominate these vertices, we must assign weight 2 to the vertices $ v_pv_jv_a $ for all $ j=1,2,...,n $. i.e., all the remaining $ n-1 $ copies of $ K_n $ in $ S_p(K_n,2) $ should have weight 2, which contradicts the fact that $ f(S_{pb}(K_n))=1$.\\
	
	\ni {\bf \emph{Case 2:}} \\
	
	\ni Let $ f(S_{pa}(K_n)) = f(S_{pb}(K_n))=f(S_{pc}(K_n))= 1 $ and $ f(S_{pj}(K_n))=2, $ for all $ j=1,2,...,n $ and $ j \neq a,b,c $. If $ a=p $ then since in $ S_{pp}(K_n) $ the vertex $ v_pv_pv_p $ is not adjacent to any vertex outside $ S_{pp}(K_n) $, it must be assigned the weight 1. Therefore, to Italian dominate $ v_pv_pv_j $, $ v_pv_jv_p $ must be assigned the weight 1, for all $ j \in \{1,2,...,n\}, j\neq p. $ In particular, $ v_pv_bv_p $ and $ v_pv_cv_p $ in $ S_{pb}(K_n) $ and $ S_{pc}(K_n) $ respectively are assigned the weight 1 each. Therefore, all the remaining vertices in $ S_{pb}(K_n) $ and $ S_{pc}(K_n) $ must be assigned weight 0. But then, $ v_pv_bv_c $ and $ v_pv_cv_b $ cannot be Italian dominated. Therefore $ a \neq p. $ Similarly we can prove that $ b,c \neq p. $ \\
	
	Since, $ f(S_{pa}(K_n)) = f(S_{pb}(K_n)) = 1,$ either $ f(v_pv_av_b) $ or $ f(v_pv_bv_a)=1 $. Without loss of generality let $ f(v_pv_av_b) = 1 $. Therefore, to Italian dominate $ v_pv_av_c $, $ f(v_pv_cv_a)=1 $ and in turn to Italian dominate $ v_pv_cv_b, $ $ f(v_pv_bv_c)=1. $ But then, to Italian dominate $ v_pv_av_p,\ v_pv_bv_p,\ v_pv_cv_p, $ $ f(v_pv_pv_a)=f(v_pv_pv_b)= f(v_pv_pv_c)=1 $ which contradicts the fact that $ f(S_{pp}(K_n))=2. $\\
	
	 Therefore, there does not exist a copy of $ S(K_n,2) $ with weight less than $ 2n-2 $ in $ S(K_n,3) $ and hence 	$ \gamma_I(S(K_n,3)) \geq 2n(n-1). $ Hence the theorem.
\end{proof}

\ni \textbf{Remark:} It is clear from the proof of the above theorem that for any $ \gamma_I$-function $f$ of $S(K_n,3)$, $ f(S_i(K_n,2))=2n-2 $, for each $ i=1,2,...,n. $ Also, in each $S_i(K_n,2)$, $f(S_{ij}(K_n))=1$ for exactly two $j$'s with $j \neq i$. (For all other $j$'s $f(S_{ij}(K_n)) =2$). For definiteness, let $f(S_{il}(K_n))= f(S_{im}(K_n)) =1 $. Note that, either all the vertices of $ S_{il}(K_n)$ or all the vertices of $ S_{im}(K_n) $ are Italian dominated by the vertices of $S_i(K_n,2)$. Let it be the vertices of $ S_{il}(K_n)$. Then to dominate the vertices of  $ S_{im}(K_n) $, we have to give non-zero weight to the vertex $v_mv_iv_i$ in $ S_{mi}(K_n)$. This process is cyclically repeated in the sense that, a vertex from $ S_i(K_n,2) $ will be Italian dominated by a vertex outside $ V(S_i(K_n,2)) $ and another vertex from $ S_i(K_n,2) $ will contribute to Italian dominate a vertex outside $ V(S_i(K_n,2)) $. Also, this process gives a partition of $ K_n $ into vertex disjoint cycles. Also note that, in $S_i(K_n,2) $, if we assign weight 1 to the extreme vertex of $S(K_n,3)$, i.e., $ v_iv_iv_i $, then with $ f(S_i(K_n,2))=2n-2 $, none of the vertices of $ S_i(K_n,2) $ can contribute to Italian dominate a vertex outside $ V(S_i(K_n,2)) $ and two of the vertices of $S_i(K_n,2)$ will be Italian dominated by vertices outside $ S_i(K_n,2) $, which will in turn increase the total weight of $ f $. So we can conclude that in any $ \gamma_I $-function of $ S(K_n,3),$ the weight of any extreme vertex is zero. \\

From the above remark we can arrive at the following corollary.

\begin{cor} \label{cor1}
	For every positive integer $n, t \geq 3$, the Italian domination number of the Sierpi\'{n}ski graph $ S(K_n,t) $ is $\gamma_I(S(K_n,t)) = n^{t-2}(2n-2). $
\end{cor}

\section{Perfect Italian Domination Number}
In this section, the exact value of the perfect Italian domination number of the Sierpinski graph, $ \gamma_I^p(S(K_n,t))$  is given. For $ n=2 $, $ S(K_n,t) = P_{2^t} $ and in \cite{Lau} it has been proved that $ \gamma_I^p(P_n)= \lceil \frac{n+1}{2}\rceil $. For $ n \geq 3 $, we have the following theorem.\\

\begin{thm}
	The perfect Italian domination number of the Sierpi\'{n}ski graph, for $ n \geq 3 $ is
	\[
	\gamma_I^p(S(K_n,t)) =
	\begin{cases}
      2n-1;\ for\ \ t=2, \\
      n^{t-2}(2n-2);\ for\ t \geq 3.
	\end{cases}
	 	\]
	
\end{thm}

\begin{proof}
	\ni{\bf \emph{Case 1:}} $ t=2 $.\\
	
\ni	Let $ V(K_n)= \{v_1,v_2,...,v_n\} $. Then $ S(K_n,2) $ is a graph with vertex set $ \{v_iv_j: i,j \in \{1,2,3,...,n\} \} $ and edge set $\{\{ v_iv_j,v_jv_i\}: v_i,v_j \in V, i\neq j \} \cup \{ \{ v_iv_j,v_iv_k \}: v_i,v_j,v_k \in V, j\neq k\}.$ Define a perfect Italian dominating function $f$ on $ S(K_n,2) $ as follows.
	\[
	f(v)=
	\begin{cases}
	1;\ v \in \{v_iv_i,v_iv_1, i=1,2,...,n\},\\
	0;\ otherwise.
	\end{cases}
	\]
	Note that, $ v_iv_i = v_iv_1$, when $ i=1. $ Therefore, in $ f $, exactly $ 2n-1 $ vertices are assigned weight 1 and all others are assigned weight 0, so that  $ f(V)=2n-1 $. Therefore, $ \gamma_I^p(S(K_n,2)) \leq 2n-1.  $ We know that $ \gamma_I(S(K_n,2))=2n-1 $ and $\gamma_I(S(K_n,2)) \leq \gamma_I^p(S(K_n,2))$. Therefore, $ \gamma_I^p(S(K_n,2)) \leq 2n-1 $. Hence $ \gamma_I^p(S(K_n,2)) = 2n-1 $.\\
	
	\ni{\bf \emph{Case 2:}} $ t \geq 3 $.\\
	
	\ni In the proof of Theorem \ref{thm1} we have defined an Italian dominating function with the property that every vertex with weight 0 is adjacent to exactly two vertices with weight 1. Therefore, $\gamma_I^p(S(K_n,t)) \leq n^{t-2}(2n-2) $. But we know that $\gamma_I(S(K_n,t)) \leq \gamma_I^p(S(K_n,t))$ and by Corollary \ref{cor1} $ \gamma_I(S(K_n,t))= n^{t-2}(2n-2)  $. Hence, $\gamma_I^p(S(K_n,t)) = n^{t-2}(2n-2). $
\end{proof}

\section{Conclusion}
In this paper we have obtained the exact value of $\gamma_I(S(K_n,t))$ and deduced $\gamma_I^p(S(K_n,t))$.  The following problems are open and are worth investigating.\\

\ni{\bf Problem 1:}  Find exact values of $ \gamma_I(S(G,t)),  $ where $G$ is any special class of graphs like $ P_n,\ C_n,\  K_{p,q} $ or tree. \\

In \cite{Jism} it has been proved that $n^{t-2}\alpha(G)\gamma_I(G) \leq \gamma_I(S(G,t)) \leq n^{t-2}(n\gamma_I(G)- |V_2|-|E_2|) $.\\

\ni{\bf Problem 2:} Find the exact value or a better bound for $ \gamma_I(S(G,t)) $, for any graph $G$.

{}

\end{document}